\documentclass[10pt]{article}
\usepackage{amsmath,amsthm,amsfonts,amssymb,amscd}

\usepackage{graphicx}
\usepackage{hyperref}
\usepackage[usenames]{color}

\newtheorem{theorem}{Theorem}[section]
\newtheorem{lemma}[theorem]{Lemma}
\theoremstyle{definition}

\newtheorem{corollary}[theorem]{Corollary}

\newtheorem{example}[theorem]{Example}

\newtheorem{remark}[theorem]{Remark}

\newcounter{comcount}

\title{On rationality of verbal subsets in a group}

\author{A. Myasnikov, V. Roman'kov}

\date{  March 22, 2011}

\begin{document}
\maketitle

\begin{abstract}
Let $F$ be a free non-abelian group. We show that for any group word $w$ the set $w[F]$ of all values of $w$ in $F$ is rational in $F$ if and only if $w[F] = 1$ or $w[F] = F$. We generalize this to  a wide class of free products of groups. 
\end{abstract}

\tableofcontents

\section{Introduction}
\label{se:intro}

In this paper we study the structure and complexity of the verbal sets in free groups and free products of groups. The main result of the paper shows that proper verbal subsets of free non-abelian groups are not rational. We generalize this to a wide class of free products of groups and other groups which have such free products as their quotients.  In particular, the result holds for arbitrary finitely generated non-abelian  residually free groups, pure braid groups,  or non-abelian right angled Artin groups.

Following Gilman \cite{Gil} we define for a given group $G$ the set $Rat(G)$ of all {\em rational subsets} of $G$ as the closure of the set of all  finite subsets of $G$ under the rational operations: union, product, and generation of a submonoid (Kleene's star operation). Sometimes, according to the standard practice, we refer to the rational subsets of  a finitely generated free monoid as {\em regular subsets}. 
It is known (see \cite{Gil}) that a subset $L$ of a group $G$ is rational in $G$ if and only if $L$ is accepted by a finite automaton over $G$ (see  definitions in Section \ref{se:pre}). 

Let  $F(X)$ be a free group with a  basis $X$ and  $W \subseteq F(X)$ a subset of $F(X)$. An element $g$ in a group $G$ is called a {\em $W$-element} if $g$ is the image in $G$ of some word $w \in W$ under some homomorphism $F(X) \to G$.  By $W[G]$ we denote the set of all $W$-elements in $G$.  The set $W[G]$ generates the {\em verbal} subgroup $W(G)$. The verbal subgroups of groups were intensely studied in group theory especially with respect to relatively free groups and varieties of groups. We refer to the books  \cite{Neumann}, \cite{Segal}  for the general facts in this area.  

{\em $W$-width} (or {\em $W$-length})  is one of the key notions  concerning verbal  subgroups $W(G)$ of a group $G$.  The $W$-length $l_W(g)$ of an element $g \in W(G)$ is the the minimal natural number $n$ such that $g$ is a product of  $n$ $W$-elements in $G$ or their inverses.  The $W$-width of the verbal subgroup $W(G)$ is $\sup\{l_W(g) \mid g \in W(g)\}$. It is usually assumed that the  set $W$ is finite. In this case the  set $W[G]$  is just a verbal set $w(G)$ for a suitable single word $w$. Furthermore, for a finite $W$ if the verbal subgroup $W(G)$ is of finite $W$-width then it is equal to $w[G]$ for some single word $w$ (not necessary from $W$). In particular, the length of $W(G)$ depends on the set $W$. From now on we will always assume that $W$ is just a singleton $W = \{w\}$ and refer to the related verbal set and the $W$-length as  $w[G]$ and $w$-length.
  Usually, it  is  very hard to compute the $w$-length of a verbal subgroup $w(G)$ or the $w$-length of  an element $g \in w(G)$. 
The first question of this type  goes back to the Ore's paper \cite{Ore} where he asked whether  the commutator length (i.e., the $[x,y]$-length) of every element in a non-abelian finite simple group is equal to 1 ({\em Ore Conjecture}). Only recently the conjecture was established by  Liebeck, O'Brian, Shalev and  Tiep \cite{LBST}. For recent spectacular results on the $w$-length in finite simple groups we refer to the papers \cite{LS, Sh1} and a  book \cite{Segal}.

Sometimes, for example in the presence of negative curvature,  it is convenient to replace the unruly standard $W$-length with a more smooth {\em stable } $W$-length,  which is defined for an element $g \in G$ as the limit $\lim_{n\to \infty}\frac{l_W(g)}{n}$ (which always exists).
In some sense the stable commutator length relates to  an $L^1$  filling norm with $\mathbb{Q}$ coefficients, introduced by Gromov in \cite{Gr1}, see also Gersten's paper \cite{Ger}. In \cite{Gr2} Gromov studied stable commutator length and its relation with  bounded cohomology. We refer to a book \cite{Cal1}  by Calegary  on stable length of elements in groups.  

On the other hand, Calegari showed in \cite{Cal2} that if a group $G$ satisfies a non-trivial law then the stable commutator length is equal to 0 for every element from $[G,G]$.  In particular, in solvable groups the stable commutator length is not very helpful, instead,  the standard $W$-length was studied intensely.  
We refer to papers \cite{Stroud},  \cite{Rhecomm}, \cite{Romankov}, and a book \cite{Segal} on  the standard width of verbal subgroups in groups.

Properties of verbal sets $W[G]$ themselves play an important part in group theory.
For example, the Membership Problem (MP) to the set $W[G]$ in $G$ is equivalent of solving in $G$ the homogeneous equations of the type $w(X) = g$, where $w \in W$ and $g \in G$. The  Endomorphism Problem in $G$, which asks to decide for  given elements $g, h \in G$ if there is an endomorphism $\phi \in End(G)$ such that $\phi(g) = h$,  is just a particular case of such a problem.  Since the Diophantine Problem (of solving arbitrary equations) is decidable in a free group $F$ \cite{Mak}, MP to the verbal sets in $F$ is also decidable.  However,   Razborov showed in \cite{Raz} that the solution sets of quadratic homogeneous equations of the type $\Pi_{i = 1}^n[x_i,y_i] = g$ may have a very complex  structure.

In this paper we focus on the complexity of verbal sets in free groups from the formal language theory view point. 
The goal is to determine the proper place of various verbal sets of a free group $F$ (and some other "free-like" groups)  in the hierarchy of formal languages.  Since  proper verbal subsets of $F$ are not rational in $F$, they are not regular in $F$, so they sit in some higher levels of the hierarchy of formal languages.  What these levels are precisely is an interesting open problem.  

The paper is organized as follows. Section 2 contains necessary definitions and facts about rational subsets in groups. In Section 3 we discuss Rhemtulla's gap theorem , which is the  main technical tool of  our approach to  verbal sets in free groups and free products of groups. 
In Section 4 we establish some key lemmas on general  properties of  rational verbal subsets, while in 
Section 5 we prove the main results  of the paper.

\section{Preliminaries}
\label{se:pre}

Let $F=F(X)$ be a free group with basis $X = \{x_1, \ldots,x_n, \ldots \}$, viewed as the set of reduced words in $X \cup X^{-1}$ with the standard multiplication. Fix $w = w(x_1, x_2, ..., x_n) \in F(X)$. An element $g$ of a group $G$ is called a  {\em $w$-element }  if $g = w(g_1, g_2, ..., g_n)$ for some  $g_1, \ldots, g_n  \in G$.  We denote the set of all $w$-elements of $G$ by $w[G]$. A subset $M \subseteq G$ is a {\em verbal subset} of $G$ if $M = w[G]$ for some word $w \in F(X).$ The subgroup $w(G)$ generated by $w[G]$ is  the $w$-{\it verbal subgroup}  of $G$, and a  subgroup of $G$ is called {\em verbal}  if it is equal to $w(G)$ for some $w$.

A word $w$ is said to be {\em proper} if there exist groups $G$ and $H$ such that $w[G] \neq  1$ and  $w[H] \neq H$, in fact, in this  case $1 \neq w[G \times H]  \neq G\times H$. 

Any element $w = w(x_1, \ldots , x_n) \in F(X)$ can be written as  a product 
$$w = x_1^{t_1}x_2^{t_2} ... x_n^{t_n} w',$$
 where $w' = w'(x_1, \ldots , x_n) \in [F,F]$. Since the exponents $t_1, t_2, ..., t_n$ depend only on the element $w$ the number $e(w) = gcd(t_1, t_2, ..., t_n)$ is well-defined (here we put $e(w) = 0$ if $t_1 = \ldots = t_n = 0$).  If $e(w) = 0$ then we refer to $w$ as a {\em commutator word}.  A non-trivial commutator word is obviously proper. If $e(w) > 0$ then  there exist integers $r_1, r_2, ..., r_n $ such that  $\sum_{i=1}^{n}r_it_i = e,$ so for an arbitrary group $G$ and an element $g \in G$ one has $w(g^{r_1}, g^{r_2}, ..., g^{r_n}) = g^e.$ In particular, $e(w) = 1$ implies that  $w[G] = G$ for every group $G,$ so $w$ is not proper.  If $e(w) > 1$ then  $w$ is  proper, which can be seen in an infinite cyclic group. In other words, a non-trivial  word  $w$ is proper if and only if $e(w) \neq 1$. 

Let $M$ be a monoid. For $L \subseteq M$ by  $L^{\ast }$  we denote the submonoid of $M$ generated by $L.$ The set
$Rat(M) \subseteq 2^M$, of all rational subsets of $M$, is defined as the smallest (with respect to inclusion) subset which contains all finite subsets of $M$ and closed under the following operations (here $L_1, L_2, L$ are subsets of $M$):

\begin{itemize}
\item{Union:} $(L_1,L_2) \to L_1 \cup L_2$. 
\item{Product:} $(L_1,L_2) \to L_1L_2 = \{ab \mid a\in L_1, b \in L_2\}$. ≈
\item{Submonoid generation:} $L \to L^{\ast }$.
\end{itemize}

It follows from the definition above  that every rational set $L \in Rat(M)$ in a monoid $M$ can be presented in a form

\begin{equation}
\label{eq:1}
L = \cup_{i=1}^{k} a_{i1}E_{i1}^{\ast } ... a_{it_i}E_{it_i}^{\ast }a_{i,t_i + 1}
\end{equation}

\noindent
where all coefficients $a_{ij}$ are in $M$ and each $E_{ij}$ is a rational subset of $M.$

We define  a {\em complexity function} $c: Rat(M) \to \mathbb{N}$  as follows. Put $c(L) = 0$ if and only if $L$ is finite.   Suppose now that  rational sets $L \in Rat(M)$ with $c(L) \leq n-1$ are defined. Then for a set   $L \in Rat(M)$ we put $c(L) = n$ if and only if $c(L) \not \leq n-1$, but either $L = L_1 \cup L_2,$ or $ L = L_1L_2,$ or $L = L_1^{\star },$ for some $L_i \in Rat(M)$ with $c(L_i) \leq n-1$, $i = 1, 2.$

It is easy to see that if $G$ is a group then for any element $g \in G$ and a set $L \in Rat(G)$ one has
\begin{equation} \label{eq:conj-rat}
c(g^{-1}Lg) = c(L).
\end{equation}

A finite $M-${\em automaton} is a tuple $A = (Q, \delta , q_0, F)$ where

\begin{itemize}
\item $Q$ is a finite set of {\em states},
\item $q_0 \in Q$ is the {\em initial} state, 
\item $F \subseteq Q$ is the set of {\em terminating} states,
\item $\delta \subseteq  Q \times M \times  Q$ is a  finite relation,  termed the {\em transition} relation.

\end{itemize}

One can view an automaton $A = (Q, \delta , q_0, F)$ as a directed $M$-labelled graph (possibly with multiple edges), with the set of vertices $Q$ and where two vertices $u, v$ are connected by an edge $u \to v$ labelled by $m \in M$ if and only if $(u,m,v) \in \delta$.  As usual, a path $p$ in $A$ from a vertex $u$ to a vertex $v$ is a sequence of edges 
$$ (u_0, m_1, u_1), (u_1, m_2, u_2), \ldots, (u_{k-1}, m_k, u_k)$$
such that $u = u_0$ and $v = u_k$. The {\em label} $\lambda(p)$  of $p$ is the product $m_1 \ldots m_k \in M$. 
A {\em successful}  path is a path from $q_0$ to a vertex in $F$. 

The subset $L(A) \subseteq M$ (the set of all elements in $M$  accepted by $A$) is defined as
$$L(A) =\{\lambda(p) \mid p \ \text{is a successful path in } A\}.$$

It was shown in \cite{Gil}  that for any monoid $M$  and any subset $L \subseteq M$ the following equivalence holds:

\[ L \in Rat(M) \Longleftrightarrow   L = L(A) \textrm{ for some automaton }  A \textrm{ over }  M.\]

Gilman showed in  \cite{Gil}  that any rational subset $L$ of a group $G$ generates in $G$ a finitely generated (and so rational) subgroup.   Hence if a verbal set $w[G]$ is rational then the verbal subgroup $w(G)$ is finitely generated.  In the case when $G$ is a free non-abelian group,  or more generally, a non-trivial free product $G = A \ast B,$ every  normal subgroup of infinite index is not finitely generated  (see \cite{MKS} and \cite{Bau}),  and so it is not rational. Hence the case  when the verbal subgroup $w(G)$ has  finite index in $G$ becomes the most interesting in our study.

\section{Free products and Rhemtulla's criterion}
\label{se:rhe1}

Let $G = A \ast B$ be a free product of  non-trivial  groups $A$ and $B$.   Each   element $u \not= 1$ of $G$ can be uniquely written in its reduced form $u = u_1u_2 ... u_m,$ where
$u_i \in A \cup B \setminus \{1\} \  (i = 1, 2, ..., m);$ and for every $i$ the elements  $u_i, u_{i+1}$  are from different groups $A$ and $B.$ The number $m$ is the syllable length of $u$, denoted by $|u|$.  
Put $supp(u) = \{u_1, \ldots,u_m\}.$   Furthermore, the reduced form of a non-trivial  element $u \in G$ can be uniquely written as 
$$
u = r_t^{-1}...r_1^{-1}v_1...v_kr_1...r_t , 
$$ 
where either  $k=1,$ or $k>1$ and  $v_1v_k \not= 1.$  We refer to  $v_1...v_k$ as the {\em core} of $u$ and denote it by ${\bar u}$.

As usual one can define a cyclically reduced form $u^{0}$ of an element $u \in G$.   Namely, let  the core ${\bar u}$ of $u$ is given in the reduced form ${\bar u} = u_1u_2 ... u_m$. Then if $u_1, u_m$ are from the different factors then $u^0 = u$. Otherwise, $u^0 = (u_2 \ldots u_{m-1}, (u_mu_1))$.

 In  \cite{Rhe}  A.H. Rhemtulla introduced  a useful technique of gap functions for $G =  A \ast B$.  To explain, suppose one of the groups, say $B$, has  an element  $b \in B$ such that $b \neq b^{-1}$.
 Let $u = u_1u_2 ... u_m$ be a non-trivial element in $ G$ given in its reduced form.
A subsequence $u_i, u_{i+1}, \ldots, u_{i+2k}$ of the reduced form of $u$ is called a $b-${\it gap} in $u$ of length $2k-1$ if $u_i = u_{i+2k} = b$ and $u_j \neq b$ for any $i < j < i+2k$.

For $k = 1, 2, \ldots$ denote by $\delta_{b,k}(u)$ the number of $b-$gaps of length $2k-1$ in $u$. 
For a positive integer $e$ put  $\gamma_{b,e} (u)$ to be the number of values $k$ such that  $\delta_{b,k}(u) \not= \delta_{b^{-1},k}(u)
(\bmod e).$ 

\medskip
{\bf Rhemtulla's criterion \cite{Rhe}.} {\it  Let $G = A \ast B$ and $b \in B$ as above. Then for any word $w(x_1, \ldots,x_n)$ with $e= e(w)  > 1$ the function $\gamma_{b,e} $ is bounded on the set $w[G].$ }

\section{Positive elements}
\label{se:pos}

A {\em sign} function on a group $G$ is a function $\rho : G \rightarrow \{-1, 1\}$ such that:

\begin{itemize}
\item $\rho (1) = 1$;
\item  for any $f,g \in G$ if  $\rho (f) = 1, \rho (g) = 1$ then $\rho (fg) = 1$.  
\end{itemize}
The set of positive elements $Pos(G)  = \{g \in G \mid \rho(g) = 1\}$ forms a submonoid in $G$. Conversely, a group $G$ with a distinguished submonoid $M \subseteq G$ admits a sign function $\rho_M$ such that  $\rho_M(g) = 1$ if and only if $g \in M$.   We refer to groups  with sign functions as {\em s-groups}. Elements from $M$ are called {\em positive}, all others - {\em negative}. The following examples are important in our context.

\begin{example}
\label{ex:cyclic}
Let $C = \langle a \rangle$ be an infinite cyclic group generated by $a$. Then $\rho(a^n) = 1$, if $n \geq 0$, and $\rho(a^n) = -1$ if $n < 0$ is a sign function. 
\end{example}

Generalizing the example above, we get the following
\begin{example} \label{ex:fg}
Let $G$ be a group with a generating  set $X.$ Then the submonoid $mon(X) = X^{\ast }$ generated by $X$ gives a sign function $\rho_X$ on $G$.   
\end{example}

\begin{example}
\label{ex:2}
Let $G = A \ast B$ be a free product of two non-trivial $s-$groups $A, B$. The  function $\rho:G \to \{-1,1\}$ such that $\rho (g) = 1 $ if and only if all factors in  the reduced form of $g \in A\ast B$ are positive, is a sign function on $G$, termed the {\em standard free product} sign function.
\end{example}
When applied to a free group $F$ with basis $X$ the examples above give the standard notion of a positive word in $F$. 

Notice, that there might be "zero divisors" in $G$ relative to $Pos(G)$, i.e., some elements $x, y \in G$, not both positive,  such that $xy \in Pos(G)$. For example, if $u \in G$ is negative then $u u^{-1}$ is positive.   To separate the natural cases like $uu^{-1}$  above (which is easy to deal with) from the harder ones we introduce the following notion. 
We  say that a sign function $\rho$ on $G$ is {\it reduced}  if   it has the following property:

\begin{itemize}

\item 
for any two subsets of elements $S, T$ of $G$ if $ST \subseteq Pos(G)$ then  here exists an element $u \in G$ such that  $Su^{-1} \subseteq Pos(G)$ and $uT \subseteq Pos(G).$

\end{itemize}

In particular, if $uT \subseteq Pos(G)$ then there is an element $u_0$ such that $uu_0^{-1} \in Pos(G)$  and  $u_0T \subseteq Pos(G)$.

Notice also, that if $\rho$ is reduced then for a fixed set $S$ there is a single $u$ that works for all subsets $T$ as above; a similar claim holds for a fixed set $T$.

We  say that a sign function $\rho$ on $G$ is {\em strongly reduced}  if   it is reduced and has the following property:

\begin{itemize}

\item 
any product of two negative elements is negative.

\end{itemize}

\begin{remark}
\label{ex:4.4}
 The sign function on the infinite cyclic group from Example \ref{ex:cyclic} is  strongly reduced.
\end{remark}

\begin{lemma}
\label{le:4.5}
 The standard sign function (see Example \ref{ex:2}) on a free product $G = A \ast B$ of two non-trivial $s-$groups with reduced sign functions is reduced.  In particular, the standard sign function on a free group is reduced. 
\end{lemma}

\begin{proof}
Let $S, T$ be non-empty subsets of $G$ such that $ST \subseteq Pos(G)$.  For an element $u \in S \setminus Pos(G)$ written in its reduced form   $u = u_1u_2 \ldots u_i \ldots u_k$ denote by  $i = i(u)$ such an index that  $u_{i}$ is negative, but all the factors $u_1, \ldots, u_{i-1}$ are positive. Similarly, for an element  $v = v_1v_2 ... v_j ... v_l \in T \setminus Pos(G)$ written in its reduced form let $j = j(v)$ be the index 
such that $v_{j}$ is negative but all the factors $v_{j+1}, \ldots, v_l$ are positive. Notice, that the sets $I = I(S) = \{|u| - i(u) +1 \mid u \in S \setminus Pos(G)   \}$ and $J = J(T) = \{ j(v) \mid v \in T \setminus Pos(G)\}$ are bounded even if the sets $S, T$ are infinite. Indeed, if, say,  $J$ is unbounded then for a given $u \in S$ there exists $v^\prime \in T$ with $|u| < j(v^\prime)$, in which case $uv^\prime$ is not positive, - contradicting the hypotheses of the lemma.  Put $max(\emptyset ) = 0,$ and $i_0 = max(I), j_0 = max(J)$. The case $S, T \subseteq Pos(G)$ when $I = J = \emptyset$ and $i_0 = j_0 = 0$ is obvious.  Assume that  $i_0  \leq j_0 $ (the other case can be treated similarly). Let  $\tilde{v} \in T \setminus Pos(G)$ be an element with  $j_0 = j(v)$.  Write $\tilde{v}$ in the reduced form $\tilde{v} = v_1 \ldots v_{j_0 -1}v_{j_0} \ldots v_l$, and assume that $v_{j_0}$ lies, say, in $B$. Denote $c = v_1 \ldots v_{j_0 -1}$.
Then for every element
$u \in S$ the factor $v_{j_0 -1}$ cancels out in the reduced form of $uv$ (otherwise a negative  factor $v_{j_0}$ occurs in the reduced form of $uv$). Hence the  reduced form of each $u \in S$ is of the type $u = u_1 \ldots u_{r(u)}c^{-1}$ for a suitable index $r(u)$ and $u_{r(u)} \in B$. Consider now two cases.

Case 1. $i_0 = j_0$. Let  $\tilde{u} \in S \setminus Pos(G)$ be an element with  $i_0 = i(u)$. In this case $\tilde{u} = \tilde{u}_1 \ldots \tilde{u}_{r(\tilde{u})}c^{-1}$, where  $\tilde{u}_{r(\tilde{u})} \in B$ is negative. The argument above shows that every element $v \in T$ can be written in the reduced form $v = cv_{j_0} \ldots v_m$, where $v_{j_0} \in B$. It follows that for any $u \in S, v \in T$ one has  $uv = u_1 \ldots (u_{r(u)}v_{j_0}) \ldots v_m $, where all the  factors, including  $(u_{r(u)}v_{j_0})$, are positive.  Since the sign function in $B$ is reduced there is an element $b \in B$ such that the elements $u_{r(u)}b^{-1}, bv_{j_0}$ are positive in $B$ for any $u \in S, v \in T$. Hence, $Scb^{-1} \subseteq Pos(G)$ and $bc^{-1}T \subseteq Pos(G)$, as required.

Case 2. $i_0 < j_0$. In this case the  reduced form of each $u \in S$ is of the type $u = u_1 \ldots u_{r(u)}c^{-1}$,  $u_{r(u)} \in B$, and furthermore, $u_{r(u)} \in Pos(B)$. Observe that for every $v \in T$ since $uv \in Pos(G)$ then in the product $c^{-1}v$ either $c^{-1}$ cancels out completely or $c^{-1}v \in Pos(G)$. In the former case $v = cv^\prime$ and the reduced form of $v^\prime$ is of the type $v^\prime_1 \ldots v^\prime_m$; in this event put  $y(v) = v^\prime_1$.  In the latter case, 
put $y(v) = 1$. By construction the  sets $S_B = \{u_{r(u)} \mid u \in S\} \subseteq B$ and $T = \{y(v) \mid  v \in T\} \subseteq B$  are such that $S_BT_B \subseteq Pos(B)$. Hence there is $b \in B$ with $S_Bb^{-1} \subseteq Pos(B), bT_B \subseteq Pos(B)$.   It follows that  $Scb^{-1} \in Post(G)$ and $bc^{-1}T \subseteq Pos(G)$ as claimed.

\end{proof}

\begin{lemma}
\label{le:4.6}
Let $A, B$ be  non-trivial groups with strongly reduced sign functions and  $G = A \ast B$ equipped with the standard free product sign function. If $L \in Rat(G)$ and  for some elements $u, v \in G$   $uLv \subseteq Pos(G)$ then $uLv \in Rat(Pos(G))$. 
\end{lemma}

\begin{proof}
 We use induction on $c(L)$. If  $c(L) = 0$ then  $L$ is finite and the claim is obvious. Now consider the following cases.

Case 1) If $L = L_1 \cup L_2$ and  $c(L_i) < c(L)$ for $i = 1,2$ the result follows by induction.  

Case 2) Suppose $L = L_1L_2$ and  $c(L_i) < c(L)$ for $ i = 1,2$. Then $uLv$ is a product of two rational sets $uL_1$ and $L_2v.$ Since the standard free product sign function on $A \ast B$ is reduced there is an element $w \in G$ for which $uL_1w^{-1}, wL_2v \subseteq Pos(G).$ Then $uL_1w^{-1}, wL_2v \in Rat(Pos(G))$ by induction, hence  $uLv = uL_1w^{-1} \cdot wL_2v \in Rat(Pos(G))$. 

Case 3) Suppose $L = L_1^{\ast },$ where $c(L_1) < c(L).$ Then $uLv = uL_1^{\ast }v = uv(v^{-1}L_1v)^{\ast },$  and $w = uv \in  Pos(G)$ since $1 \in L_1^{\ast }$. Denote  $L_2 = v^{-1}L_1v$,  so $uLv  = wL_2^\ast$ and   $c(L_2) = c(v^{-1}L_1v) = c(L_1) < c(L)$.  

If $L_2 \subseteq Pos(G)$ then the  result follows by  induction. Otherwise, there is an element $l \in L_2 \smallsetminus Pos(G)$. If $l = l_1 ... l_i l_{i+1} ... l_t$ is the reduced form of $l$  then there exists $i$ such that  $l_i$ is  negative and all the factors $l_{i+1}, \ldots, l_t$  are positive.  Suppose that the reduce form of $l$ is written in the form 
$$l = r_t^{-1}...r_1^{-1}u_1...u_kr_1...r_t ,$$  
where either  $k=1,$ or $k>1$ and  $u_1u_k \not= 1.$  

Claim 1. In the notation above the factor $l_i$ of $l$ is among  the first $t+1$ factors $r_t^{-1}, \ldots, r_1^{-1}, u_1$, i.e., $i \leq t+1$.

Indeed, the reduced form of every element $d \in wL_2^{\ast}$ must have the product $l_{i-1}^{-1}...l_1^{-1}$ at the end, otherwise the element  $dl \in wL_2^{\ast}$ would not be  positive. 

If $i > |l|/2$ then for every $d \in wL_2^{\ast}$ more then half of $l$ cancels out in $dl$, so $|dl| < |d|$. This implies that $wL_2^{\ast}$ does not have an element of minimal length, i.e.,  $wL_2^{\ast}  = \emptyset$, which contradicts the fact that  $w \in wL_2^{\ast}$. 

Assume now that $i \leq |l|/2$, but $i > t+1$. It follows that $k \geq 2$.  In this case for any   natural number $p$ one has $l^p \in L_2^\ast$  and the rightmost occurrence of  the negative factor $l_i$ of $l$ does not cancel in $l^p$ (since it is in the core $\bar l$ of $l$).  Therefore,  for sufficiently large $p$ the rightmost occurrence of the factor $l_i$ does not cancel in $wl^p$, so the element $wl^p$ is negative, which  contradicts the condition $wL_2^\ast \subseteq Pos(G)$. Hence $i \leq t+1$ and the claimed follows.

Let's take $l  \in L_2 \smallsetminus Pos(G)$ such that $i = i(l)$ be the maximal possible such index   among  all elements $l$ in $L_2 \smallsetminus Pos(G)$ (such $i$  exists since $wL_2^\ast \subseteq Pos(G)$). We can also assume  that $l_i = b \in B.$  It follows that the reduced form of $w$ is equal to $w'b(w)l_{i-1}^{-1} ... l_1^{-1}$, where $b(w) \in B$ and  $b(w)b \in Pos(B)$.  

Claim 2. The reduced form of any element $d \in L_2$  can be written as  
$d = d'b(d)l_{i-1}^{-1} ... l_1^{-1},$ where $b(d) \in B$ and $b(d)b \in Pos(B).$  

  Indeed, suppose   $m \in L_2$  does not have  $l_{i-1}^{-1}...l_1^{-1}$ at the end.   Then  $m = m_1...m_s l_q^{-1}...l_1^{-1},$ where $0 \leq q < i-1$ (we assume $m = m_1...m_s$ for $q = 0$) and  $m_sl_{q+1} \not= 1$ (notice that $m_s$ and $l_{q+1}$ are in the same factor).  Since $wml$ is positive the negative factor  $l_i$ of $l$ cancels out in $wml$, so $m_1...m_{s-1}(m_sl_{q+1})  $ must cancel out in $wm$.  The element $w$ ends on $l_{i-1}^{-1}...l_1^{-1}.$  If $s >q$ then $m = l_1...l_qm_{q+1} ... m_s l_{q}^{-1}...l_1^{-1},$  and  $m_{q+1} = l_{q+1}$. On the other hand  $m_sl_{q+1} \neq 1$, which shows that $m_{q+1} ... m_s$ is the core of $m$.  If $s \neq q+1$ then the length of the core of $m$ is greater than 1, so for sufficiently large integer $p$  the element $wm^pl \in wL_2^\ast$  is negative - contradiction. If $s = q+1$  then  $m = l_1 ... l_q l_{q+1}l_q^{-1}...l_1^{-1}$ and $wml = w'l_{q+1}^{-1}l_{q+1}l_{q+1}l_{q+2} \ldots l_t$, whose reduced form is $w'l_{q+1}l_{q+2} \ldots l_t$, so it contains $l_i$ - contradiction.  The case $s\leq q$ can be done similarly. This proves the claim. 
  
  In the notation above, since the sign function on $B$ is strongly reduced there is a positive element $b_0 \in B$ for which $b_0b,$ $b(w)b_0^{-1}$ and $b(d)b_0^{-1}$ are positive for all $d \in L_2.$ Denote $r = b_0l_{i-1}^{-1} ... l_1^{-1}.$ Notice, that $r$ is positive.   It follows from Claim 2 that every element $d \in L_2$  can be written in the reduced form as  
$d = d'(b(d)b_0^{-1})r.$ Let 
$$
L_3 = \{d'(b(d)b_0^{-1}) \mid d \in L_2\}
$$
  
  Claim 3. The language $rL_3$ is positive.
  
 Indeed, it follows from the argument above.

 Notice that $rL_3 = rL_2r^{-1}$ so  $rL_3$ is rational and  $c(rL_3) = c(L_2) < c(L)$. Hence by induction $rL_3 \in Rat(Pos(G)).$ Now  

\begin{equation}
 \label{eq:2}
wL_2^{\ast } = w'r\{L_2\}^{\ast } = w'\{rL_3\}^{\ast } r,
\end{equation}
  
\noindent
where all factors in product on the right are positive.   Hence $uLv  = wL_2^\ast  \in Rat(Pos(G))$, which proves the lemma.

\end{proof}

\begin{lemma}
\label{le:4.7}
Let $A, B$ be  non-trivial groups with  strongly reduced sign functions and  $G = A \ast B$ equipped with the standard sign function. If $L \in Rat(G)$ and  $L \subseteq Pos(G)$ then $L \in Rat(Pos(G))$.   
\end{lemma}
\begin{proof}
 We use induction on  complexity $c(L)$ of $L$.  If $c(L) = 0$ then $L$ is finite and the claim is obvious.  
Consider the following cases. 

Case 1) If $L = L_1 \cup L_2$ or $L = L_1^\ast$ with $c(L_i) < c(L), i = 1,2,$ then $L_1, L_2 \subseteq Pos(G)$ and the result follows by induction.  

Case 2) Suppose $L = L_1L_2$.  By Lemma \ref{le:4.5} the standard free product sign function on $G$ is 
reduced. Therefore there is $u \in G$ such that $L_1u, u^{-1}L_2 \subseteq Pos(G)$. By Lemma \ref{le:4.6} 
$L_1u, u^{-1}L_2 \subseteq Rat(Pos(G))$.  Hence $L = (L_1u)(u^{-1})L_2 \subseteq Rat(Pos(G))$, as claimed.

   \end{proof}

\begin{lemma}
\label{le:4.8}
Let $A, B$ be two non-trivial groups with  strongly reduced sign functions and such that the sets $Rat(A), Rat(B)$ are closed under intersections and complements (form Boolean algebras) and  $G = A \ast B$. If the submonoid  $Pos(G)$ relative to the standard sign function on $G$ is rational then for any  $\bar{L} \in Rat(G)$ the intersection $L = \bar{L} \cap Pos(G)$  is rational in the monoid $Pos(G).$
\end{lemma}

\begin{proof} Observe, that under the premises of the theorem  $Rat(G)$ is a Boolean algebra by G.A. Bazhenova's result \cite{Baz1}, which states that class of groups with Boolean algebras of rational subsets is closed under free products.  Hence $L$ is rational in $G.$ Then $L \in Rat(Pos(G))$ by Lemma \ref{le:4.7}.

\end{proof}

\begin{corollary}
\label{co:4.9}
Let  $F_2 = F(X_2)$ be a free non-abelian group of  rank $2$ with   basis $X_2 = \{x_1, x_2\}.$
Denote by $X_2^{\ast } $ the free submonoid of $F_2$ generated by $X_2.$ 
If  $\bar{L} \in Rat(F_2)$ then $L = \bar{L} \cap X_2^{\ast } \in Rat(X_2^{\ast })$. 
\end{corollary}

\begin{proof} Follows from Lemma \ref{le:4.8}.

 \end{proof} 

\section{Free groups and free products}
\label{se:fre}

Let   $F = F(X)$ be a free non-abelian group with  basis $X = \{x_1, x_2, ... \}.$
Put $X_2 = \{x_1, x_2\}$ and consider the free group $F_2 = F(X_2)$ with basis $X_2$ as the distinguished subgroup of $F(X)$ generated by $X_2$.  
By $X_2^{\ast }$ we denote  the free submonoid of $F_2$ generated by $X_2.$

\begin{lemma}
\label{le:5.1}
Let $G = A \ast B$ be a free product of two non-trivial groups $A$ and $B.$  If  $w$  is a  proper word such that the verbal subgroup $w(G)$  has infinite index in $G$  then the set $w[G]$ is not rational in $G$.
\end{lemma}

\begin{proof}  Observe that for such $w$  the verbal subgroup $w(G)$ is a non-trivial normal subgroup of infinite index in $G$. By B. Baumslag's  result \cite{Bau} the subgroup $w(G)$ is not finitely generated. Since a subgroup generated in a group by a  rational subset has to be finitely generated (see \cite{Gil}, Theorem 4.2), the generating set  $w[G]$ of $w(G)$ is not rational in $G$.
\end{proof}

\begin{corollary}
\label{co:5.2}
Let $G = A \ast B$ be a free product of non-trivial groups $A$ and $B$ with infinite abelianization $G_{ab} = G/[G, G].$
Then for any non-trivial commutator word $w$ the set $w[G]$ is not rational in $G$.
\end{corollary}

\begin{corollary}
\label{co:5.3}
 Let $F$ be a free non-abelian group. Then for any non-trivial commutator word $w$ the set $w[F]$ 
 is not rational in $F.$
\end{corollary}

For the rest of the paper we fix a proper non-commutator word $w$. Observe, that $e = e(w) \geq 2$.
To apply Rhemtulla's criterion we view the free group $F_2$ as a free product 
$F_2 = \langle x_1 \rangle \ast \langle x_2 \rangle   = A \ast B,$ equipped with the standard sign function given by the submonoid $X_2^{\ast }.$ This sign function is  strongly reduced.

\begin{lemma}
\label{le:5.4}
 Let $p,q \in X_2^{\ast }$ and $E \subseteq X_2^{\ast }$ be such that $pE^{\ast }q \subseteq w[F_2]$. Then one of the following  hold:
 \begin{itemize}
 \item [1)] $|u| \geq 2$  for every $u \in E^\ast$. In this case $supp(pE^{\ast }q)$ is finite.
  \item [2)] $|u| = 1$ for every $u \in E^\ast$. In this case either $E^\ast \subseteq x_1^\ast$ or $E^\ast \subseteq x_2^\ast$. 
\end{itemize}

\end{lemma}

\begin{proof} Let $u, v \in  E^\ast$. Since $E \subseteq X_2^{\ast }$ the elements $u, v$ are positive, so they are equal to their cores $u = {\bar u}, v = {\bar v}$.  Assume that ${\bar u}, {\bar v}$ are given in the reduced forms
$${\bar u} = u_1u_2 ... u_k,  \ \ \ {\bar v} = v_1v_2 ... v_l. $$
We prove first that if  $k, l  \geq 2$  then  $supp(u^0) = supp(v^0)$. 
Notice, that either $supp(u^0) = \{u_1, u_2, ..., u_k\}$ for even $k,$ or $supp(u^0) = \{u_2, u_3, ..., (u_ku_1)\}$ for odd $k$.   

Since 
\begin{equation}
 \label{eq:4}
puu\{v\}^{\ast}uuq \subseteq   X_2^{\ast } \cap w[F_2] 
\end{equation}
 the Rhemtulla's criterion shows that $supp(u^0) \subseteq supp(v^0).$  Indeed, $|\bar v|\geq 2$ implies that the length of $v^n$ strictly grows with $n$, so if $b$ is a factor in $supp(u^0)$ but not in $supp(v^0)$ then for infinitely many $k$ there is a number $n = n(k)$ such that the word $puuv^nuuq$ contains precisely one  $b$-gap of length $k$. So the gap function $\delta_{b,k}(u)$ is equal to 1 for infinitely many $k$ on $puu\{v\}^{\ast}uuq $. Notice that $\delta_{b^{-1},k}(puuv^nuuq) = 0$ since the words are positive. Hence the function   $\gamma_{b,e} $ is unbounded on   $w[F_2] $ - contradicting the Remtulla's criterion.  Similarly, we show that 
 $supp(v^0) \subseteq supp(u^0)$. Hence $supp(u^0) = supp(v^0)$ as claimed.

Observe now that if there are elements $u, v \in E $ such that $|u| \geq 2$ and $|v| = 1$ then for sufficiently large $n$ one has $|u|, |uv^n| \geq 2$ and  $supp(u^0) \neq supp((uv^n)^0)$ - contradicting the statement above. 

The argument above shows that either all elements in $E^{\ast}$ are of syllable length greater then $1,$ or all of them have length $1.$  This proves 2) and the first part of 1). To finish the proof observe that  for any 
$u = u_1\ldots u_k \in E^\ast$ one has  $u_2, \ldots, u_{k-1}, u_ku_1 \in supp(v^0)$. Since    $u_1,u_k, u_1u_k$ are positive there are only finitely many choices for $u_1$ and $u_k$ as divisors of $u_ku_1$.
This proves that there is a finite set $K \subseteq X_2^\ast$ such that for any $u \in E^\ast$ $supp(u) \subseteq K$. It follows that $supp(pE^{\ast }q)$ is finite as claimed.

\end{proof}

\begin{corollary}
 \label{co:5.5}
Let $L \in Rat(X_2^{\ast}) \cap w[F_2]$. Then there is a finite set $K_L \subseteq x_1^\ast \cup x_2^\ast$ and a natural number $n = n(L)$ such that every element  $u \in L$ can be presented as a product of the following type:
$$
u = s_1t_1s_2 \ldots s_nt_n,
$$
where $supp(s_i) \subseteq K_L$ and $t_i \in x_1^\ast \cup x_2^\ast$.  
\end{corollary}

\begin{proof}
Since $L \in Rat(X_2^{\ast})$  it can be presented in the form 
$$L= a_{i1}E_{i1}^{\ast } \ldots a_{ij}E_{ij}^{\ast } a_{ij+1} \ldots E_{it_i}^{\ast}a_{it_i+1} $$
For $E_{ij}$ as above, put $E = E_{ij}$ and denote  $p = a_{i1} \ldots a_{ij}, q = a_{ij+1} \ldots a_{it_i+1} \in X_2^{\ast }.$ Since each $E_{il}$ contain $1$ we have $pE^{\ast}q \subseteq Rat(X_2^{\ast}) \cap w[F_2]$. Now the result follows from  Lemma \ref{le:5.4}.
\end{proof}

\begin{theorem}
\label{the:5.6} \label{co:5.7}
Let $F$ be a free non-abelian group  and $w$ be a  proper word. Then the set  $w[F]$ 
is not rational in $F.$
\end{theorem}

\begin{proof} Let $w$ be a proper word such that $w[F]$ is rational in $F$. Notice that $e = e(w) \geq 2$ since $w$ is proper. By Corollary \ref{co:5.3} the word $w$ is not a commutator word. Observe that $\bar{L} = w[F] \cap F_2 = w[F_2]$ is rational in $F_2$ as a homomorphic image of a rational set under standard homomorphism $F \rightarrow F_2.$    By Corollary \ref{co:4.9} the set $L =
\bar{L} \cap X_2^{\ast }$ is rational in $X_2^{\ast }.$   By Corollary \ref{co:5.5} there is  a finite set $K_L \subseteq x_1^\ast \cup x_2^\ast$ and a natural number $n = n(L)$ such that every element  $u \in L$ can be presented as a product of the following type:
\begin{equation}\label{eq:siti}
u = s_1t_1s_2 \ldots s_nt_n,
\end{equation}
where $supp(s_i) \subseteq K_L$ and $t_i \in x_1^\ast \cup x_2^\ast$.  Chose  $t \in \mathbb{N}$ large enough so $x_1^t \not \in K_L$. Chose $l \in \mathbb{N}$ such that $l > n(L)$. Then the word 
 $u = (x_1^tx_2)^{le}$ belongs to $w[F_2]$, hence it belongs to $L$. However,  $u$ cannot be presented in the form (\ref{eq:siti}) - contradiction, which  proves the theorem.

\end{proof}

Theorem  \ref{co:5.7} can be generalized into free products as follows.

\begin{theorem}
\label{th:5.8}
Let $A$ and $B$ be groups containing elements of infinite order $x_1 \in A$, $x_2 \in B,$ and $G = A \ast B.$  If  the rational
sets $Rat(A)$ and $Rat(B)$ are Boolean algebras then  for every proper word $w$ with $e(w) \geq 2$ the set $w[G]$ is not rational in $G.$
\end{theorem}

\begin{proof}  Notice first that by  Bazhenova's
result \cite{Baz1} the set $Rat(G)$ is a Boolean algebra. 
Obviously, the subgroup  generated by $x_1$ and $x_2$ in $G$  is a free subgroup $F_2$ with basis $\{x_1, x_2\}$. If  $w[G] \in Rat(G)$  
  then $\bar{L} = w[G] \cap F_2$ is rational in $G$ by \cite{Baz1}, hence by another   Bazhenova's result \cite{Baz2} $\bar{L}$ is rational in $F_2$.   By Corollary \ref{co:4.9} the set $L =
\bar{L} \cap X_2^{\ast }$ is rational in $X_2^{\ast }  = \{x_1,x_2\}^\ast$ and so has a presentation of the form (\ref{eq:1}) in $X_2^{\ast }.$ Since $w[F_2] \subseteq w[G]$ and the free decompositions of $F_2$ is induced from the free decomposition of $G$ one can complete the  proof by an 
argument similar to the one from the proof of Theorem \ref{the:5.6}.

\end{proof}

\begin{corollary}
 \label{co:5.9}
Let $A$ and $B$ be groups containing elements of infinite order, and $G = A \ast B.$ Let the abelianization $G_{ab}$ is infinite, and the rational sets $Rat(A)$ and $Rat(B)$ are Boolean algebras. Then for every proper word $w$ the set $w[G]$ is not rational in $G.$ 
\end{corollary}

\begin{proof}
 Follows from Corollary \ref{co:5.2} and Theorem \ref{th:5.8}.
\end{proof}

\begin{corollary}
 \label{co:5.9b}
Let $A$ and $B$ be infinite finitely generated abelian groups  and $G = A \ast B.$ Then for every proper word $w$ the set $w[G]$ is not rational in $G.$ 
\end{corollary}
\begin{proof}
Bazhenova showed in  \cite{Baz2} that rational sets in finitely generated abelian groups form Boolean algebras. Now the result follows from  \ref{co:5.9}.
\end{proof}

Theorem \ref{th:5.8} and Corollaries \ref{co:5.9}  and \ref{co:5.9b} have far reaching generalizations. To explain we need the following simple but useful result.

\begin{lemma}
\label{th:5.10}
Suppose that a  group $H$ admits a homomorphism onto a group $G$ in which every set $w[G]$ for a  proper word $w$ is not rational.   Then for every proper  word $w$ the set $w[H]$ is not rational in $H.$

\end{lemma}

\begin{proof} Suppose that $\varphi $ is a homomorphism of $H$ onto $G.$ Since for every word $w$ we have $w[G]  = \varphi (w[H]),$ and a homomorphic image of any rational set is rational (see \cite{Gil}) $w[H] \in  Rat(H)$ implies that $w[G] \in Rat(G)$ that contradicts our assumption. Hence, $w[H] \not\in Rat(H).$

\end{proof}

\begin{corollary}
\label{co:5.11}
Suppose that a  group $H$ admits a homomorphism onto a  free non-abelian group $F$.   Then for every proper  word $w$ the set $w[H]$ is not rational in $H.$

\end{corollary}

There are many classes of groups which have free non-abelian quotients. We list some of them below.

\begin{corollary}
\label{co:5.11b}
In the following groups $H$  for every proper  word $w$ the set $w[H]$ is not rational:
\begin{itemize}
\item [1)] Pure braid groups $PB_n$ for $n \geq
 3$.
\item [2)] Non-abelian right angled Artin groups.
\item [3)] Finitely generated non-abelian residually free groups.
\end{itemize}
\end{corollary}
  \begin{proof} To prove 1) observe  that
  a pure braid group $PB_n, n\geq 3,$ has the group $PB_3$ as its epimorphic
  quotient  (see \cite{Birman}, for example), and the group $PB_3$ is isomorphic to $F_2 \times \mathbb{Z}$,
  so $PB_n, n \geq 3,$ has the free group $F_2$ as its quotient.

To see 2) Let $G = G(\Gamma)$ be  a non-abelian partially commutative group
  corresponding to a finite graph $\Gamma$. Then
there are three vertices in $\Gamma$, say $v_1, v_2, v_3$ such that
the complete subgraph $\Gamma_0$ of $\Gamma$ generated by these
vertices is not a triangle. In particular, a partially commutative
group $G_0 = G(\Gamma_0)$ is either a free group $F_3$ (no edges in
$\Gamma_0$), or $(\mathbb{Z} \times \mathbb{Z})  \ast \mathbb{Z}$
(only one edge in $\Gamma_0$), or $F_2 \times \mathbb{Z}$ (precisely
two edges in $\Gamma_0$). Notice that in all three cases the group
$G(\Gamma_0)$ has $F_2$ as its epimorphic quotient. Now, it suffices
to show that $G(\Gamma_0)$ is an epimorphic quotient of $G(\Gamma)$,
which is obtained from $G(\Gamma)$ by adding  to the standard
presentation of $G(\Gamma)$ all the relations of the type $v = 1$,
where $v$ is a vertex of $\Gamma$ different from $v_1, v_2, v_3$.
This shows that $F_2$ is a quotient of $G(\Gamma)$ 

By definition every non-abelian residually free group has a free non-abelian quotient, so 3) holds. 
\end{proof}

The following result shows that a wide class of free products of groups does not have rational proper verbal subsets. 

\begin{corollary}
\label{co:free-product-quotient}
Suppose that    $H = C \ast D$, where the factors  $C, D$ are finitely generated groups with infinite abelianizations.   Then for every proper  word $w$ the set $w[H]$ is not rational in $H.$

\end{corollary}


\begin{thebibliography}{\hspace{0.5in}}

\bibitem{Baz1} Bazhenova, G.A.:  Closure of one class of groups under free products.  Siberian  Math. J.,  41,  611-613 (2000).
\bibitem{Baz2} Bazhenova, G.A.:  On rational sets in finitely generated nilpotent groups.  Algebra and Log., 39, 215-223 (2000). 


\bibitem{Bau}  Baumslag, B.:  Intersections of finitely generated subgroups in free products. J. London Math. Soc., 41,
673-679 (1966). 

\bibitem {Birman}  Birman, J.:  Braids, Links and Mapping Class Groups.
 Princeton University Press, Annals of Math. Studies, Princeton (1974).

\bibitem{Cal1}  Calegari, D.:  SCL. Math. Soc. Japan  Mem., vol. 20, Tokyo (2009).

\bibitem{Cal2}  Calegari, D.:   Quasimorphisms and laws. Algebr. Geom. Topol., 10, 215-217 (2010).


\bibitem{Ger}  Gersten, S.:  Cohomological lower bounds for isoperimetric functions on groups. Topology,  37,  1031-1072 (1998).

\bibitem{Gil}  Gilman, R.H.:  Formal languages and infinite groups. In: Geometric and computational perspectives of infinite groups (Minneapolis, MN and New Brunswick, NJ, 1994), vol. 25 of DIMACS Ser. Discrete Math. Theoret. Comput. Sci., 27-51. Amer. Math. Soc., Providence, RI (1996).


\bibitem{Gr1}  Gromov, M.:  Volume and bounded cohomology. IHES Publ. Math., 56, 5-99 (1982).


\bibitem{Gr2}  Gromov, M.:  Asymptotic invariants of infinite groups.
In:  London Math. Soc.  Lect. Notes Ser., vol. 182, Cambridge (1993).

\bibitem{LS}  Larsen, M.  and   Shalev, A.:  Word maps and Waring type problems. J. Amer. Math. Soc. 22,  437-466 (2009).

\bibitem{LBST} Liebeck,M.,  O'Brian, E.,  Shalev A.  and   Tiep, P.:  The Ore conjecture. J. European Math. Soc., 12,  939-1008 (2010).

\bibitem{MKS}  Magnus, W.,  Karrass, A. and  Solitar, D.:  Combinatorial Group Theory. Wiley Interscience, New York (1968). 

\bibitem{Mak} Makanin, G.S.:   Equations in a free group (Russian), Izvestia Akademii Nauk SSSR, Ser. Matemat.,  46, 1199-1273 (1982).

\bibitem{Neumann}  Neumann, H.:  Varieties of groups. Springer-Verlag, New York (1967).

\bibitem{Ore}  Ore, O.:  Some remarks on commutators. Proc. Amer. Math. Soc. 2, 307-314 (1951).

\bibitem{Raz}  Razborov, A.:  On the parametrization of solutions for equations in free groups. Int. J. Algebra  and Comput.,  3, 251-273 (1993).

\bibitem{Rhe}   Rhemtulla, A.H.:  A problem of bounded expressability in free products. Proc. Cambridge Philos. Soc., 64, 573-584 (1968).


\bibitem{Rhecomm} Rhemtulla, A.H.:  Commutators of certain finitely generated soluble groups.  Canad. J. Math., 21, 1160-1164 (1969).

\bibitem{Romankov}  Roman'kov, V.A.:  Width of verbal subgroups in solvable groups. Algebra and Log., 21, 41-49 (1982).

\bibitem{Segal}  Segal, D.:  Words: notes on verbal width in groups, London Math. Soc. Lect. Notes Ser., vol. 361, Cambridge Univ. Press, Cambridge (2009).

\bibitem{Sh1}  Shalev, A.:  Word maps, conjugacy classes, and a noncommutative Waring-type theorem. Ann. of Math.,  170, 1383-1416 (2009).

\bibitem{Stroud} Stroud, P.W.:  Topics in the theory of verbal subgroups.  PhD Thesis, Univ. of Cambridge, Cambridge (1966).






\end{thebibliography}
\end{document}